\documentclass{amsart}
\usepackage{amsmath,amsthm,amssymb,IMjournal}

\usepackage{graphics,graphicx}
\usepackage{mathrsfs, bm}
\usepackage{color}

\begin{document}
\newtheorem{theorem}{Theorem}[section]
\newtheorem{dfn}{Definition}[section]
\newtheorem{lemma}{Lemma}[section]
\newtheorem{remark}{Remark}[section]

\numberwithin{equation}{section}

\title{ESTIMATES TO THE WEAK SOLUTION OF THE ELECTRO-HYDRODYNAMICAL BOUNDARY VALUE PROBLEM FOR THE UNIT CELL OF CATION-EXCHANGE MEMBRANE. THE CASE OF NONZERO DEBYE RADIUS.}

\author{Yulia Koroleva
\\ Department of Higher Mathematics,  HSE (National Research University) }



\abstract
  We study a model problem on the filtration of a conducting fluid through a porous layer. A porous medium is presented as an
assemblage of identical spherical cells. Each cell consists of a porous core and liquid shell. We show the dependence of each flow parameter on the Debye radius  which characterizes how far the influence of a charge extends in electrolyte. The common case of finite Debye radius in comparison to the cell radius is analyzed. We derive apriori estimates for flow characteristics which show the specific behavior of the fluid. The boundedness of velocity field, pressure, electric potential and ion flux densities was proved. 

\endabstract




\section{Introduction}\label{sec1}
The current research is devoted to a filtration process through a porous layer. This topic was investigated by physicists and chemists.

In particular, professor A.N. Filippov has proposed a new method for calculating the density of solvent – $U,$ solute – $J,$ and electric current – $I$ fluxes through an ion-exchange membrane under the simultaneous action of external pressure gradients $p,$ chemical $\mu$ (electrolyte concentration $C$), and electric potential $\varphi$ (see  \cite{Filippov1}). In \cite{Filippov1}, the cell model of an ion-exchange membrane consisting of porous charged particles-balls of the same radius is constructed, the problem of finding the kinetic coefficients $L_{ij}$ of the Onsager matrix is posed and solved in general, and an exact algebraic formula for the hydrodynamic permeability $L_{11}$ of the membrane is obtained. In \cite{Filippov2}, the electroosmotic permeability $L_{12}$ and the specific electrical conductivity $L_{22}$ of the ion exchange membrane were calculated. In \cite{Filippov3}, new formulas are obtained for the integral diffusion permeability $L_{33}$ and electrodiffusion coefficient $L_{23}$ of a charged membrane in equilibrium with an aqueous solution of a binary electrolyte. The cell model was successfully verified using experimental data on the electrical conductivity and electroosmotic permeability of an aqueous solution of hydrochloric acid through the pristine cation-exchange membrane MF-4SC and that modified by halloysite nanotubes and platinum or iron nanoparticles \cite{Filippov4}. It is shown that with an increase in the equilibrium concentration of the electrolyte, the total permeability of the membrane also increases due to both barofiltration and electroosmotic transfer of the solvent. However, in \cite{Filippov1}--\cite{Filippov4}, the electric double layer (EDL) at the interface between a porous particle and a pure liquid was effectively replaced by jumps in the electric potential and concentration when crossing this boundary. THe mentioned results was done under the assumption that the thikness of electric double layer is small in comparison with the radius of the particle.

In present paper, we investigate a more general problem on analysis of flow characteristics in the case when the thickness of the electric double layers is not neglected.  This approach leads to analysis of the Poisson-Boltzmann equation in conjunction with the Stokes equations for a conducting viscous incompressible liquid (electrolyte solution) and the Brinkman equations for a porous medium. The mentioned system of equations can not be solved analytically in general. However,  it was possible to obtain estimates for the liquid velocity, ion flux density, electric potential, and concentration, depending on the ratio between the dimensionless thickness of the electric double layer, the Peclet number, and other parameters of the boundary value problem set for a single cell of the membrane.
Let us mention that the problem studied in the present research was completely new. Some  related statements of problems were investigated in papers \cite{rel1} --\cite{rel3}. The existence of global weak solutions for the Nernst-Planck-Poisson problem which
describes the evolution of concentrations of charged species, subject to Fickian diffusion
and chemical reactions in the presence of an electrical field was proved in \cite{rel1}. The paper \cite{rel2} considers ionic electrodiffusion in fluids, described by the Nernst–Planck–Navier–Stokes system. It was proved that the system has global a smooth solutions for
arbitrary smooth data in bounded domains with a smooth boundary in three space
dimensions for the following cases: an arbitrary positive Dirichlet boundary conditions for the ionic concentrations, arbitrary Dirichlet boundary
conditions for the potential, arbitrary positive initial concentrations, and arbitrary
regular divergence-free initial velocities. The parameter Debye length characterizes how far the influence of a charge extends in electrolyte. The analysis of the limit of vanishing Debye length for ionic diffusion in fluids, described by the Nernst–Planck–Navier–Stokes system was carried out in \cite{rel3}. In the asymptotically stable
cases of blocking (vanishing normal flux) and uniform selective (special Dirichlet)
boundary conditions for the ionic concentrations, the authors proved that the ionic charge
density  converges in time to zero in the interior of the domain, in the limit of
vanishing Debye length.
The current research is aimed to consider all possible cases for nonzero Debye radius. The main new investigations are as follows. Depending on the Debye radius, the norms of electric potential are estimated via the norms of concentrations. The obtained estimates show that the bigger the parameter $\delta,$ the less the influence of concentration. The effect of concentration is not much significant when the Debye radius $\delta$ is much bigger than size of outer layer. 

Based on the estimates of electric potential, we have obtained the bounds for the pressure and velocity.
The main result also is the bounds for the hydrodynamic permeability $L_{11}$ and ion flux densities depending on the Debye radius, geometrical characteristics of the porous cell and the velocity of the incoming flow. We have analyzed the dependence of flow characteristics on Debye radius and on the other parameter Peclet number. It compares the rate of heat transfer via convection to heat transfer via diffusion. Moreover, the ratio between Debye radius and Peclet number was analyzed.

\section{Statement of the problem}

Consider a porous cell $\Omega=\Omega^i\cup\Omega^o\in \mathbb{R}^3$ which  is a spherical particle:
$$\Omega^i=\{ 0<r<a, \varphi\in[0,2\pi], \theta\in[-\pi,\pi]\},\ \Omega^o=\{ a<r<b, \varphi\in[0,2\pi], \theta\in[-\pi,\pi]\}.$$

Its boundary is denoted by $\partial\Omega=\Gamma^i\cup\Gamma^o,$ where
$$\Gamma^i=\{r=a\}, \Gamma^o=\{r=b\}.$$  
Flow of fluid (electrolyte) in the outer domain $\Omega^o$ can be described by Stokes equation under low Reynold's number and which involves also electromass force:
\begin{equation}\label{St_eq}
\nabla {{p}^{o}}={{\text{ }\!\!\mu\!\!\text{ }}^{o}}\Delta {{\mathbf{v}}^{o}}-{{\text{ }\!\!\rho\!\!\text{ }}^{o}}\nabla {{\text{ }\!\!\varphi\!\!\text{ }}^{o}},
\end{equation}																								
where
\begin{equation}\label{ro}
{{\text{ }\!\!\rho\!\!\text{ }}^{o}}={{F}_{0}}\left( {{Z}_{+}}C_{+}^{o}-{{Z}_{-}}C_{-}^{o} \right)
\end{equation}	
is the  volumetric density of movable electric charges in a porous particle,  ${{Z}_{\pm }}$ are charge modules of cations and anions of the electrolyte, $C_{\pm }^{o}$ are concentrations of cations and anions; ${{F}_{0}}$ is Faraday constant.
 We assume that the liquid is incompressible:
 \begin{equation}\label{divvo}
\nabla \cdot {{\mathbf{v}}^{o}}=0.																														
\end{equation}
Here ${{p}^{o}}$ is local pressure, ${{\mathbf{v}}^{o}}$ is velocity vector, ${{\text{ }\!\!\mu\!\!\text{ }}^{o}}$ is a dynamic viscosity, ${{\text{ }\!\!\varphi\!\!\text{ }}^{o}}$ is the electric potential which satisfies to the Poisson equation:
\begin{equation}\label{Pois_out}
\Delta {{\text{ }\!\!\varphi\!\!\text{ }}^{o}}\text{=}-\frac{{{\text{ }\!\!\rho\!\!\text{ }}^{o}}}{\text{ }\!\!\varepsilon\!\!\text{ }{{\text{ }\!\!\varepsilon\!\!\text{ }}_{0}}},
\end{equation}																													
where $\text{ }\!\!\varepsilon\!\!\text{ }$ is the  relative  permittivity of the medium, ${{\text{ }\!\!\varepsilon\!\!\text{ }}_{0}}$is the dielectric constant.
In general stationary case system (\ref{St_eq})--(\ref{Pois_out}) must contain equations of charge conservations:
\begin{equation}\label{J}
\nabla \cdot \mathbf{J}_{\pm }^{o}=0,
\end{equation}																													
where $\mathbf{J}_{\pm }^{o}$ are ions flux densities and which can be written due to  the following Nernst representation:
$$\mathbf{J}_{\pm }^{o}={{\mathbf{v}}^{o}}C_{\pm }^{o}-{{D}_{\pm }}\left( \nabla C_{\pm }^{o}\pm {{Z}_{\pm }}C_{\pm }^{o}\nabla {{\text{ }\!\!\varphi\!\!\text{ }}^{o}}\frac{{{F}_{0}}}{RT} \right).$$																		
Here ${{D}_{\pm }}$ is the coefficient of ion's diffusion in the fluid, $R$ is the gas constant, $T$ is the absolute temperature.

Fluid flow is subjected to the Brinkman's equation with mass electric force in the inner domain  $\Omega^i:$
\begin{equation}\label{Brinkman}\left( a<r<b \right)
\nabla {{p}^{i}}={{\text{ }\!\!\mu\!\!\text{ }}^{i}}\Delta {{\mathbf{v}}^{i}}-{{\text{ }\!\!\rho\!\!\text{ }}^{i}}\nabla {{\text{ }\!\!\varphi\!\!\text{ }}^{i}}-k{{\mathbf{v}}^{i}},	
\end{equation}																						
where
\begin{equation}\label{rho}
{{\text{ }\!\!\rho\!\!\text{ }}^{i}}={{F}_{0}}\left( {{Z}_{+}}C_{+}^{i}-{{Z}_{-}}C_{-}^{i} \right)	
\end{equation}																								
is the volumetric  density of movable electric charges in a porous particle, $k$ is the Brinkman's constant which is inversely proportional to the particle's permeability. Here ${{\text{ }\!\!\mu\!\!\text{ }}^{i}}$ is the coefficient of  Brinkman's viscosity.  Brinkman's liquid is assumed to be incompressible as well:
\begin{equation}\label{divvi}
\nabla \cdot {{\mathbf{v}}^{i}}=0.
\end{equation}																														
Electric potential satisfies to Poisson's equation:
\begin{equation}\label{Pois_in}
\Delta {{\text{ }\!\!\varphi\!\!\text{ }}^{i}}\text{=}-\frac{\left( {{\text{ }\!\!\rho\!\!\text{ }}^{i}}-{{\text{ }\!\!\rho\!\!\text{ }}_{V}} \right)}{\text{ }\!\!\varepsilon\!\!\text{ }{{\text{ }\!\!\varepsilon\!\!\text{ }}_{0}}},						
\end{equation}																				
where ${{\text{ }\!\!\rho\!\!\text{ }}_{V}}$ is the bulk density of the fixed charges. We model our membrane such that the particle has negative charge, then ${{\text{ }\!\!\rho\!\!\text{ }}_{V}}>0$. For the convenience of our analysis we will assume the equality of permittivity for the liquid and Brinkman's medium in Poisson's equations (\ref{Pois_in}) and (\ref{Pois_out}).  This condition lets us do not take into account the Maxwel's stress tensor in boundary conditions since it keeps continuous automatically under the considered case. Besides already stated equations, the following  charge conservation equations must be valid:
\begin{equation}\label{Ji}
\nabla \cdot \mathbf{J}_{\pm }^{i}=0,
\end{equation}																														
where $\mathbf{J}_{\pm }^{i}$ are  densities of ion's fluxes in the porous particle,
$$\mathbf{J}_{\pm }^{i}={{\mathbf{v}}^{i}}C_{\pm }^{i}-{{D}_{m\pm }}\left( \nabla C_{\pm }^{i}\pm {{Z}_{\pm }}C_{\pm }^{i}\nabla {{\text{ }\!\!\varphi\!\!\text{ }}^{i}}\frac{{{F}_{0}}}{RT} \right).$$																	
Here ${{D}_{m\pm }}$ are diffusion coefficients of electrolyte ions inside the porous particle.
\subsection{Boundary conditions}

We assume the following natural boundary conditions.

I.\  On the border  $r=a$ one must have:
\begin{itemize}
\item	Continuity of velocity field:   ${\mathbf {v}}^o={\mathbf {v}}^i;$
\item	Continuity of stress tensor: $\Sigma^o=\Sigma^i;$
\item	Continuity of the radial component of ion's fluxes:
$\mathbf{{J}}_{\pm }^{\text{o}}\cdot \mathbf{n}=\mathbf{J}_{\pm }^{\text{i}}\cdot \mathbf{n}.$
\item	Continuity of the electric potential $\varphi $ and both components of the electric field strength vector: $\varphi^o=\varphi^i,$ $\nabla \varphi^o=\nabla\varphi^i;$
\end{itemize}
II.\ On the boundary  $r=b$ it is required:\\
\begin{itemize}
\item	Cunningham conditions on the velocity:
${\mathbf {v}}^o={\mathbf {U}};$
\item	Zero  gradient of ion concentrations (chemical potentials):
$${{\left. C_{\text{ }\!\!\pm\!\!\text{ }}^{\text{o}} \right|}_{r=b,\theta =0}}={{\left. C_{\text{ }\!\!\pm\!\!\text{ }}^{\text{o}} \right|}_{r=b,\theta =\pi }}.$$
\item	Zero  gradient of electric potential:
$${{\left. {{\varphi }^{\text{o}}} \right|}_{r=b,\theta =0}}={{\left. {{\varphi }^{\text{o}}} \right|}_{r=b,\theta =\pi }}$$
\end{itemize}
\subsection{Boundary-value problems in the dimensionless form}
For a given incoming velocity ${\mathbf {U}} $ we denote its absolute value by $U:=|{\mathbf {U}} |.$ Let $U_0$ be a characteristic value of the filtration velocity.
We make the change of variables to get all variables and constants in dimensionless form:
$$\tilde{r}=\frac{r}{a},\,\mathbf{\tilde{v}}=\frac{{\mathbf{v}}}{{U_0}},\,\tilde{p}=\frac{p}{{p}_{0}},\,\widetilde{\sigma}_{r\theta}=\frac{{\sigma}_{r\theta}}{p_0},\,\widetilde{\sigma}_{rr}=\frac{{\sigma}_{rr}}{p_0},\,\widetilde{C}_{\pm}=\frac{C_{\pm}}{C_0},
\widetilde{\varphi}=\frac{\varphi F_0}{RT},$$
$${{\mathbf{j}}_{\pm }}\,=\frac{{{\mathbf{J}}_{\pm }}}{{{U_0}}{{C}_{0}}},\tilde{\nabla }=a\nabla ,\,\,\tilde{\Delta }={{a}^{2}}\Delta ,\, \nu_{\pm}=\frac{D_0}{D_{\pm}},\  \nu_{m\pm}=\frac{D_0}{D_{m\pm}},\ \text{Pe}=\frac{a{{U_0}}}{{{D}_{0}}},\ \sigma=\frac{\rho_V}{F_0 C_0}$$
$${{p}_{0}}=RT{{C}_{0}},\,{{U_0}}=\frac{a p_0}{\mu^o},\ \delta=\frac{d}{a},\ m=\frac{\mu^i}{\mu^o},\ s^2=\frac{a^2k}{\mu^i},\ s_0^2=ms^2=\frac{a^2}{R_b^2},$$
where ${{R}_{b}}=\sqrt{\frac{\mu^o}{k}}$ is the Brinkman's radius (characteristic thickness of the filtration layer), ${{C}_{0}}={{C}_{e\pm }}{{Z}_{\pm }}$ is the equivalent concentration of electrolyte equilibrated with a membrane, ${{D}_{0}}$ is diffusion coefficient scale, $d=\left(\frac{C_0F_0^2}{\varepsilon\varepsilon_0 RT}\right)^{-\frac{1}{2}}$ is the Debye radius.
Passing finally to dimensionless variables  (we shall omit tilda sign do not heavy the notations), we got the following systems of differential equations which describe our filtration process:  \\


\begin{equation}\label{Pr_in}
\begin{aligned}
&\nabla {{p}^{\kappa}}=m_{\kappa}\Delta {{\mathbf{v}^{\kappa}}}-\left( {{Z}_{+}}C^{\kappa}_{+}-{{Z}_{-}}C^{\kappa}_{-} \right)\nabla {{\text{ }\!\!\varphi^{\kappa}\!\!}}-{\mu}_{\kappa}{{\mathbf{v}^{\kappa}}},\\
&\nabla \cdot {{\mathbf{v}^{\kappa}}}=0,\\
&{{\text{ }\!\!\delta\!\!\text{ }}^{2}}\Delta {{\text{ }\!\!\varphi^{\kappa}\!\!\text{ }}}=-\left( {{Z}_{+}}C^{\kappa}_{+}-{{Z}_{-}}C^{\kappa}_{-}-\text{ }\!\!\sigma_{\kappa}\!\!\text{ } \right)\\					
&\nabla \cdot \mathbf{j}^{\kappa}_{\pm }=0,\\
&\mathbf{j}^{\kappa}_{\pm }={{\mathbf{v}^{\kappa}}}C^{\kappa}_{\pm }-\frac{1}{{{\text{ }\!\!\nu^{\kappa}\!\!\text{ }}}\text{Pe}}\left( \nabla C^{\kappa}_{\pm }\pm {{Z}_{\pm }}C^{\kappa}_{\pm }\nabla {{\text{ }\!\!\varphi^{\kappa}\!\!\text{ }}} \right).
\end{aligned}
\end{equation}

Here $${m}_{\kappa}=\begin{cases}& 1, \text{ if } \kappa=o,\\ & m, \text{ if } \kappa=i, \end{cases}\quad \quad
{\mu}_{\kappa}=\begin{cases}& 0, \text{ if } \kappa=o,\\ & ms^2, \text{ if } \kappa=i, \end{cases}$$
$${\sigma}^{\kappa}=\begin{cases}& 0, \text{ if } \kappa=o,\\ & \sigma, \text{ if } \kappa=i, \end{cases}\quad \quad
{\nu}^{\kappa}=\begin{cases}& {\nu}_{\pm}, \text{ if } \kappa=o,\\ & {\nu}_{m\pm }, \text{ if } \kappa=i. \end{cases}$$

Equations for boundary conditions are modified analogously according to the new introduced variables.

 On the border  $r=1$ one must have:
 \begin{itemize}
\item	Continuity of velocity field:   ${\mathbf {v}}^o={\mathbf {v}}^i;$
\item	Continuity of stress tensor: $\Sigma^o=\Sigma^i;$
\item	Continuity of the radial component of ion's fluxes:
$\mathbf{{j}}_{\pm }^{\text{o}}\cdot \mathbf{n}=\mathbf{j}_{\pm }^{\text{i}}\cdot \mathbf{n}.$
\item	Continuity of the electric potential $\varphi $ and both components of the electric field strength vector: $\varphi^o=\varphi^i,$ $\nabla \varphi^o=\nabla\varphi^i;$
\end{itemize}
On the outer boundary  $r=\frac{1}{\gamma}$ it is required:\\
\begin{itemize}
\item	Cunningham conditions on the velocity:
${\mathbf {v}}^o={\mathbf {U}};$
\item	Zero  gradient of ion concentrations:
$${{\left. C_{\text{ }\!\!\pm\!\!\text{ }}^{\text{o}} \right|}_{r=\frac{1}{\gamma},\theta =0}}={{\left. C_{\text{ }\!\!\pm\!\!\text{ }}^{\text{o}} \right|}_{r=\frac{1}{\gamma},\theta =\pi }}.$$
\item	Zero  gradient of electric potential:
$${{\left. {{\varphi }^{\text{o}}} \right|}_{r=\frac{1}{\gamma},\theta =0}}={{\left. {{\varphi }^{\text{o}}} \right|}_{r=\frac{1}{\gamma},\theta =\pi }}$$
\end{itemize}

\section{Preliminaries}

All unknown functions are considered in the weak sense, i.e. as general functions from Sobolev space $H^1$ satisfying a corresponding integral identity.
Let us recall that $H^1(\Omega)$ is a closure of $C^{\infty}$ functions with respect to the $H^1-$ norm
$$\|f\|^2_{H^1}(\Omega):=\int\limits_{\Omega}f^2\,dx\,dy+\int\limits_{\Omega}|\nabla f|^2\,dx\,dy.$$
The notation
${\bf f}\in H^1(\Omega;\mathbb{R}^3)$ means that each component of the  vector-function ${\bf f}=(f_1,f_2,f_3)$ belongs to the same space:  $f_i\in
H^1(\Omega).$ The notation $H^1(\Omega,\Gamma;\mathbb{R}^3)$ is reserved for all functions from the space $H^1(\Omega;\mathbb{R}^3)$ having a zero trace on the set $\Gamma\subset
\partial\Omega.$ Let us denote further the outer normal vector to the boundary $\Gamma$ by ${\bf n}_{\Gamma}=(n^1_{\Gamma},n^2_{\Gamma}, n^3_{\Gamma}).$
For a given vector-function ${\bf f}(x_1,x_2,x_3)=(f_1,f_2,f_3)$ we use the operator $\nabla:$
$$|\nabla \mathbf f|^2=\sum\limits_{i,j=1}^3\frac{\partial f_i}{\partial x_j}\frac{\partial f_i}{\partial x_j}$$
In particular, ${\bf n}_{\partial \Omega_i}$ means the outer normal vector to the domain $\Omega_i.$
Define the spaces $$A^{i}=\{{\bf u}\in H^1(\Omega;\mathbb{R}^3):\quad \mathrm{div}{\bf \mathbf u}=0\}$$ and
$$A^{o}=\{{\bf u}\in H^1(\Omega;\mathbb{R}^3): \quad \mathrm{div}{\bf \mathbf u}=0,{\bf \mathbf u}={\bf \mathbf U}\text{
on } \Gamma^o\}.$$

The following Friedrich's-type inequality holds for functions from $H^1(\Omega;\mathbb{R}^3):$
\begin{equation}\label{Friedr}
\int\limits_{\Omega}{\bf f}^2\,dx\leq d^2_{\Omega}\left(\int\limits_{\Omega}|\nabla{\bf f}|^2\,dx+\int\limits_{\partial\Omega}{\bf f}^2\,dS\right),
\end{equation}
where $d_{\Omega}$ is the diameter of $\Omega.$


\section{Weak solution}

The solution to the system  (\ref{Pr_in}) together with the mentioned boundary conditions is understood in the weak sense, i.e. as a solution to integral identities with an appropriate test-function.

\begin{dfn}
The functions ${\bf v}^{\kappa}\in H^1(\Omega;\mathbb{R}^3)\cap A_1,$  $p^{\kappa}\in H^1(\Omega),$ $\varphi^{\kappa}\in H^1(\Omega),$  $C_{\pm}^{\kappa}\in H^1(\Omega),$  ${\bf j}_{\pm}^{\kappa}\in H^1(\Omega;\mathbb{R}^3)\cap A_1,$   are the weak
solution to (\ref{Pr_in}) iff the following integral identities hold:
\begin{equation}\label{weakvi}
\begin{aligned}
&\int\limits_{\Gamma^{\kappa}}p^{\kappa}{\bf u}^{\kappa}{\bf n}_{\Gamma^{\kappa}}\,dS=-m_{{\kappa}}\int\limits_{\Omega^{\kappa}}\nabla{\bf v}^{\kappa}:\nabla{\bf u}^{\kappa}\,dx+m\int\limits_{\Gamma^{\kappa}}{\bf u}^{\kappa}\frac{\partial {\bf v}^{\kappa} }{\partial{\bf n}_{\Gamma^{\kappa}}}\,dS-\\
&-\int\limits_{\Omega^{\kappa}}\left(Z_{+}C^{\kappa}_{+}-Z_{-}C^{\kappa}_{-}\right)\nabla{\varphi^{\kappa}}{\bf u}^{\kappa}\,dx-ms^2\int\limits_{\Omega^{\kappa}}{\bf v}^{\kappa}{\bf u}^{\kappa}\,dx
\end{aligned}
\end{equation}
for any test-function ${\bf u}^{\kappa}\in H^1(\Omega^{\kappa};\mathbb{R}^3)\cap A_1;$
\begin{equation}\label{weakfii}
\begin{aligned}
&-\delta^2 \int\limits_{\Omega^{\kappa}}\nabla{\varphi^{\kappa}}\nabla { \psi}^{\kappa}\,dx+
\delta^2\int\limits_{\Gamma^{\kappa}}{\psi^{\kappa}}\frac{\partial {\varphi^{\kappa}} }{\partial{\bf n}_{\Gamma^{\kappa}}}\,dS=\\
&=  -\int\limits_{\Omega^{\kappa}}\left(Z_{+}C^{\kappa}_{+}-Z_{-}C^{\kappa}_{-}-\sigma_{\kappa}\right){ \psi}^{\kappa}\,dx
\end{aligned}
\end{equation}
{ for all } ${\psi}^{\kappa}\in H^1(\Omega^{\kappa});$

\begin{equation}\label{divvi}
\int\limits_{\Omega^{\kappa}}{\bf v}^{\kappa}\nabla q^{\kappa}\,dx=\int\limits_{\Gamma^{\kappa}}{ {{\bf v}^{\kappa}}}q^{\kappa}\cdot {\bf n}_{\Gamma^{\kappa}}\,dS,\quad \text{ for all } q^{\kappa}\in L_2(\Omega^{\kappa});
\end{equation}
\begin{equation}\label{divji}
\int\limits_{\Omega^{\kappa}}{\bf j}_{\pm}^{\kappa}\nabla \chi^{\kappa}\,dx=\int\limits_{\Gamma^{\kappa}}{ {{\bf j}_{\pm}}^{\kappa}}\chi^{\kappa}\cdot {\bf n}_{\Gamma^{\kappa}}\,dS,\quad \text{ for all } \chi^{\kappa}\in L_2(\Omega^{\kappa});
\end{equation}

\begin{equation}\label{j0}
\begin{aligned}
&\int\limits_{\Omega^{\kappa}}{\bf j_{\pm}}^{\kappa}{\bf \xi}^{\kappa}\,dx=\int\limits_{\Omega^{\kappa}}{\bf v}^{\kappa}{\bf \xi}^{\kappa}C^{\kappa}_{\pm}\,dx-\frac{1}{\nu_{{\kappa}}Pe}\int\limits_{\Gamma^{\kappa}}
\frac{\partial{ C_{\pm}}^{\kappa}}{\partial{\bf n}_{\Gamma^{\kappa}}}{\bf \xi}^{\kappa}\,dS-\\
&-\frac{1}{\nu_{{\kappa}}Pe}\int\limits_{\Omega^{\kappa}}Z_{\pm}C^{\kappa}_{\pm}\nabla \varphi^{\kappa}{ \bf \xi}^{\kappa}\,dx
 \end{aligned}
\end{equation}
{ for all } ${\bf \xi}^{\kappa}\in H^1(\Omega^{\kappa};\mathbb{R}^3)\cap A_1.$
\end{dfn}

\begin{remark}
If one chose $q^i=q^o=1$ in (\ref{divvi}), then the incoming velocity $U$ and the velocity filed on the common part $\Gamma^i$ satisfies the integral equation
\begin{equation}\label{bondaryU}
\int\limits_{\Gamma^o}{\bf U}n_{\Gamma^o}\,dS=\int\limits_{\Gamma^i}{\bf v}^in_{\Gamma^i}\,dS=0.
\end{equation}
Analogously,  equations
\begin{equation}\label{bondaryj}
\int\limits_{\Gamma^o}{\bf j}^o_{\pm}n_{\Gamma^o}\,dS=\int\limits_{\Gamma^i}{\bf j}^i_{\pm} n_{\Gamma^i}\,dS=0
\end{equation}

hold. These facts are consequence of classical Ostrogradsky-Gauss theorem for a divergence-free vector field.
In addition, Nernst formula and divergence-free  of vectors $\bf v^{\kappa},$ ${\bf j}^{\kappa}_{\pm}$ imply
\begin{equation}\label{bondaryCfi}
{\bf v}^\kappa\nabla C^\kappa_{\pm}-\frac{1}{\nu_{\kappa}Pe}\Delta C^\kappa_{\pm}\mp\frac{Z_{\pm}}{\nu_{\kappa}Pe}\nabla \varphi^\kappa\nabla C_{\pm}^\kappa-\frac{Z_{\pm}}{\nu_{\kappa}Pe}\Delta \varphi^\kappa C_{\pm}^\kappa=0,
\end{equation}
 where $$\kappa=i,o,\quad \nu_{\kappa}=\begin{cases}&\nu, \text{ if } \kappa=o,\\
 &\nu_{m}, \text { if } \kappa=i.
 \end{cases}$$

\end{remark}

\section{Apriori estimates for velocity, pressure and electric potential.}
Let us chose ${\bf u}^o={\bf v}^o$ as a test-function in (\ref{weakvi}). Due to boundary conditions we deduce
\begin{equation}\label{weakv00}
\begin{aligned}                 
&\int\limits_{0}^{2\pi}p^o|_{r=\frac{1}{\gamma}} U\cos\theta\,d\theta-\int\limits_{0}^{2\pi}p^o|_{r=1}{v_r^i}\,d\theta=-\int\limits_{\Omega^o}|\nabla{\bf v}^o|^2\,dx+\frac{1}{2}\int\limits_{0}^{2\pi}\frac{\partial {\bf U}^2 }{\partial r}|_{r=\frac{1}{\gamma}}\,d\theta-\\
&-\frac{1}{2}\int\limits_{0}^{2\pi}\frac{\partial ({\bf v}^i)^2} {\partial r}|_{r=1}\,d\theta
-\int\limits_{\Omega^o}\left(Z_{+}C^o_{+}-Z_{-}C^o_{-}\right)\nabla{\varphi^o}{\bf v}^o\,dx
\end{aligned}
\end{equation}

Applying (\ref{weakvi}) with ${\bf u}^i={\bf v}^i$ we obtain the identity

\begin{equation}\label{weakvii}
\begin{aligned}
&\frac{m}{2}\int\limits_{0}^{2\pi}\frac{\partial {|{\bf v}^i|^2} }{\partial r}|_{r=1}\,d\theta-\int\limits_{0}^{2\pi}p^i|_{r=1}{v_r^i}\,d\theta=m\int\limits_{\Omega^i}|\nabla{\bf v}^i|^2\,dx+\\
&+\int\limits_{\Omega^i}\left(Z_{+}C^i_{+}-Z_{-}C^i_{-}\right)\nabla{\varphi^i}{\bf v}^i\,dx+ms^2\int\limits_{\Omega^i}|{\bf v}^i|^2\,dx
\end{aligned}
\end{equation}

Summing (\ref{weakv00}) with (\ref{weakvii}) divided by $m$ and taking into the account boundary conditions we arrive at the identity
\begin{equation}\label{weakvsum}
\begin{aligned}
&\int\limits_{\Omega^o}|\nabla{\bf v}^o|^2\,dx+s^2\int\limits_{\Omega^i}|{\bf v}^i|^2\,dx+\int\limits_{\Omega^i}|\nabla{\bf v}^i|^2\,dx+\\
&+\int\limits_{\Omega^o}\left(Z_{+}C^o_{+}-Z_{-}C^o_{-}\right)\nabla{\varphi^o}{\bf v}^o\,dx+\int\limits_{\Omega^i}\left(Z_{+}C^i_{+}-Z_{-}C^i_{-}-\sigma\right)\nabla{\varphi^i}{\bf v}^i\,dx+\\
&\int\limits_{0}^{2\pi}p^o|_{r=\frac{1}{\gamma}} U\cos\theta\,d\theta+(m+1)\int\limits_{0}^{2\pi}\frac{\partial {({\bf v}^i)^2} }{\partial r}|_{r=1}\,d\theta=\frac{m-1}{m}\int\limits_{0}^{2\pi}p^i{v_r^i}|_{r=1}\,d\theta
\end{aligned}
\end{equation}

\
Consider now (\ref{weakvi}) and (\ref{weakfii}) with $\psi^o=\varphi^o,$ $\psi^i=\varphi^i.$ Taking into the account boundary conditions, one get
\begin{equation}\label{weakfi0o}
\begin{aligned}
&-\delta^2 \int\limits_{\Omega^o}|\nabla{\varphi^o}|^2\,dx+\frac{\delta^2}{2}\int\limits_{0}^{2\pi}\frac{\partial {|{\bf \varphi}^o|^2} }{\partial r}|_{r=\frac{1}{\gamma}}\,d\theta-
\frac{\delta^2}{2}\int\limits_{0}^{2\pi}\frac{\partial {|{\bf \varphi^i}|^2} }{\partial r}|_{r=1}\,d\theta=\\
&=  -\int\limits_{\Omega^o}\left(Z_{+}C^o_{+}-Z_{-}C^o_{-}\right){ \varphi}^o\,dx
\end{aligned}
\end{equation}
and
\begin{equation}\label{5.5}
\begin{aligned}
&-\delta^2 \int\limits_{\Omega^i}|\nabla{\varphi^i}|^2\,dx+
\frac{\delta^2}{2}\int\limits_{0}^{2\pi}\frac{\partial {|{\bf \varphi}^i|^2} }{\partial r}|_{r=1}\,d\theta=\\
&=  -\int\limits_{\Omega^i}\left(Z_{+}C^i_{+}-Z_{-}C^i_{-}-\sigma\right){ \varphi}^i\,dx
\end{aligned}
\end{equation}

Since the function $\varphi^o$ is defined up to the constant, we will look for the solution among the functions with zero mean on the boundary:
\begin{equation}\label{zeromeanfi}
\int\limits_{\partial\Omega^o}{\varphi^o}\,dx=0.
\end{equation}
This assumption implies the validity of Friedrich's-type estimate:
\begin{equation}\label{Frfi}
\int\limits_{\Omega^o}|{\varphi^o}|^2\,dx\leq d^2_{\Omega^o}\int\limits_{\Omega^o}|\nabla{\varphi^o}|^2\,dx, \quad \int\limits_{\Omega^i}|{\varphi^i}|^2\,dx\leq d^2_{\Omega^i}\int\limits_{\Omega^i}|\nabla{\varphi^i}|^2\,dx+d^2_{\Omega^i}\int\limits_{\Gamma^i}|{\varphi^i}|^2\,dx,
\end{equation}
where $d_{\Omega^o},d_{\Omega^i}$ are diameters of the corresponding domains $\Omega^o,$ $\Omega^i.$

Summing identities (\ref{weakfi0o}) and (\ref{5.5}) it is possible to estimate the norms of electric potential via norms of concentrations:

\begin{equation}\label{normsfi}
\begin{aligned}
&\delta^2\left(\int\limits_{\Omega^i}|\nabla {\varphi^i}|^2\,dx+\int\limits_{\Omega^o}|\nabla {\varphi^o}|^2\,dx\right)=\\
&\int\limits_{\Omega^i}\left(Z_{+}C^i_{+}-Z_{-}C^i_{-}-\sigma\right){ \varphi}^i\,dx+\int\limits_{\Omega^o}\left(Z_{+}C^o_{+}-Z_{-}C^o_{-}\right){ \varphi}^o\,dx\leq \\
&\frac{1}{\kappa_1}\int\limits_{\Omega^i}\left(Z_{+}C^i_{+}-Z_{-}C^i_{-}-\sigma\right)^2\,dx+\kappa_1\int\limits_{\Omega^i}\left({ \varphi}^i\right)^2\,dx+\\
&\frac{1}{\kappa_2}\int\limits_{\Omega^o}\left(Z_{+}C^o_{+}-Z_{-}C^o_{-}\right)^2\,dx+\kappa_2\int\limits_{\Omega^o}\left({ \varphi}^o\right)^2\,dx.
\end{aligned}
\end{equation}
where constants $\kappa_i>0$  will be chosen in the sequel. Applying Friedrich's inequalities one gets

 \begin{equation}\label{normsfiL_2}
\begin{aligned}
&(\delta^2d^{-2}_{\Omega^i}-\kappa_1)\|\varphi^i\|^2_{L_2(\Omega^i)}+(\delta^2d^{-2}_{\Omega^o}-\kappa_2)\|\varphi^o\|^2_{L_2(\Omega^o)}\leq \\
&\frac{1}{\kappa_1}\|Z_{+}C^i_{+}-Z_{-}C^i_{-}-\sigma\|^2_{L_2(\Omega^i)} +\frac{1}{\kappa_2}\|Z_{+}C^o_{+}-Z_{-}C^o_{-}\|^2_{L_2(\Omega^o)}
\end{aligned}
\end{equation}
and
 \begin{equation}\label{normsfigradL_2}
\begin{aligned}
&(\delta^2-\kappa_1d^2_{\Omega^i})\|\nabla \varphi^i\|^2_{L_2(\Omega^i)}+(\delta^2-\kappa_2d^{-2}_{\Omega^o})\|\nabla \varphi^o\|^2_{L_2(\Omega^o)}\leq\\
&\frac{1}{\kappa_1}\|Z_{+}C^i_{+}-Z_{-}C^i_{-}-\sigma\|^2_{L_2(\Omega^i)} +\frac{1}{\kappa_2}\|Z_{+}C^o_{+}-Z_{-}C^o_{-}\|^2_{L_2(\Omega^o)}.
\end{aligned}
\end{equation}
Thus,
 \begin{equation}\label{normsfiH1}
\begin{aligned}
&C_{\varphi^i}\|\varphi^i\|^2_{H^1(\Omega^i)}+C_{\varphi^o}\|\varphi^o\|^2_{H^1(\Omega^o)}\leq  \\
&\frac{1}{\kappa_1}\|Z_{+}C^i_{+}-Z_{-}C^i_{-}-\sigma\|^2_{L_2(\Omega^i)} +\frac{1}{\kappa_2}\|Z_{+}C^o_{+}-Z_{-}C^o_{-}\|^2_{L_2(\Omega^o)},
\end{aligned}
\end{equation}
where constants $$C_{\varphi^i}=\min\{\delta^2d^{-2}_{\Omega^i}-\kappa_1,\delta^2-\kappa_1d^2_{\Omega^i}\}, \quad C_{\varphi^o}=\min\{\delta^2d^{-2}_{\Omega^o}-\kappa_2,\delta^2-\kappa_2d^2_{\Omega^o}\}.$$
Fix now $\kappa_1=\frac{\delta^2}{2d^2_{\Omega^i}},$ $\kappa_2=\frac{\delta^2}{2d^2_{\Omega^o}}.$ With this choice one deduce

\begin{equation}\label{normsfiH1constants}
\begin{aligned}
&\frac{\delta^2}{2}\|\varphi^i\|^2_{H^1(\Omega^i)}+\frac{\delta^2}{8(1-1/{\gamma})^2}\|\varphi^o\|^2_{H^1(\Omega^o)}\leq \\
&\frac{2}{\delta^2}\|Z_{+}C^i_{+}-Z_{-}C^i_{-}-\sigma\|^2_{L_2(\Omega^i)} +\frac{8(1-1/{\gamma})^2}{\delta^2}\|Z_{+}C^o_{+}-Z_{-}C^o_{-}\|^2_{L_2(\Omega^o)},
\end{aligned}
\end{equation}

\begin{remark}
Observe that the terms with norms of concentrations in the right-hand side of (\ref{normsfiH1constants}) are neglectful as $\delta\rightarrow\infty,$ i.e. the bigger $\delta$ the less the influence of concentration. Moreover, if $\delta\gg\left(1-\frac{1}{\gamma}\right),$ then the las term in the right-hand side of (\ref{normsfiH1constants}) tends to zero, thus, the influence of concentrations in the outer domain is small.
\end{remark}
\section{Analysis of the systems}

Combining the first and third equations of system (\ref{Pr_in}), one derives that
$$\nabla p^o=\triangle \mathbf{v}^{o}+\delta^2\triangle \varphi^o\nabla \varphi^o;$$
Analogously,
$$\nabla p^i+\sigma\nabla \varphi^i=m\triangle \mathbf{v}^{i}+ms^2\mathbf{v}^{i}+\delta^2\triangle \varphi^i\nabla \varphi^i;$$

Let us consider an axillary boundary-value problem
\begin{equation}\label{auxv}
\begin{cases}
&\nabla p^o=\triangle \mathbf{\bf v}^{o}+\Phi^o, \text { in } \Omega^o,\\
&\nabla (p^i+\sigma\varphi^i)=m\triangle \mathbf{\bf v}^{i}-ms^2{\bf v}^i+\Phi^i, \text { in } \Omega^i,\\
&\nabla\cdot{\bf v}^o=0,\text { in } \Omega^o,\\
&\nabla\cdot{\bf v}^i=0,\text { in } \Omega^i,\\
&{\bf v}^o={\bf U}, \text { on } \Gamma^o,\\
&{\bf v}^i={\bf v}^o, \text { on } \Gamma^i,\\
\end{cases}
\end{equation}
where $\Phi^o\in L_2(\Omega^o;\mathbb{R}^3),\Phi^i\in L_2(\Omega^i;\mathbb{R}^3),$ ${\bf U}\in L_2(\Gamma^o;\mathbb{R}^3).$

One can show in a standard way (see e.g. \cite{Lions}) the existence of the weak solution to this problem. The next Lemmas prove some apriori bounds for the velocity field and pressure. 

\begin{lemma}\label{L1}
The norms of functions ${\bf v}^o,{\bf v}^i\in H^1(\Omega)$  are bounded and satisfy inequalities
\begin{equation}\label{normsv}
\begin{aligned}
&\|{\bf v}^o\|^2\leq C(\|\Phi^o\|^2_{L_2}+\|{\bf U}\|^2_{L_2}+\|{\bf U}^o\|^2_{L_2});\\
&\|{\bf v}^i\|^2\leq C(\|\Phi^i\|^2_{L_2}+\|{\bf U}^i\|^2_{L_2}),
\end{aligned}
\end{equation}
where ${\bf U}^i, {\bf U}^o$ are divergence-free functions such that
$${\bf U}^i|_{\Gamma^i}={\bf v}^o|_{\Gamma^i}, \quad {\bf U}^o|_{\Gamma^o}={\bf U}, {\bf U}^o|_{\Gamma^i}={\bf v}^i|_{\Gamma^i}.$$
Moreover, considering $p^o$ such that
\begin{equation}\label{pressure_mean}
\int\limits_{\partial\Omega^o}p^o\,dx=0
\end{equation}

 the pressure is uniquely defined and satisfy estimates
\begin{equation}\label{normspressure}
\begin{aligned}
&\|\nabla p^o\|^2_{H^{-1}}+\|\nabla {\bf v}^o\|^2_{L_2}\leq C_3\|\Phi^o\|^2_{L_2}, \quad \| p^o\|^2_{L_2}< C_3d^2_{\Omega^o}\|\Phi^o\|^2_{L_2}\\
&\|\nabla (p^i+\sigma \varphi^i)\|^2_{H^{-1}}+m\|\nabla {\bf v}^i\|^2_{L_2}+ms^2\| {\bf v}^i\|^2_{L_2}\leq C_4\|\Phi^i\|^2_{L_2}, \quad \| p^i\|^2_{L_2}< C_3d^2_{\Omega^i}\|\Phi^i\|^2_{L_2}
\end{aligned}
\end{equation}
\end{lemma}

\begin{proof}
There exist  divergence-free vector-functions ${\bf U}^i\in H^1(\Omega^i;\mathbb{R}^3)$  and ${\bf U}^o\in H^1(\Omega^o;\mathbb{R}^3)$ such that
$${\bf U}^i|_{\Gamma^i}={\bf v}^o|_{\Gamma^i}, \quad {\bf U}^o|_{\Gamma^o}={\bf U}, {\bf U}^o|_{\Gamma^i}={\bf U}^i|_{\Gamma^i}.$$
the functions ${\bf v}^{o}-{\bf U}^o$ and ${\bf v}^{i}-{\bf U}^i$ vanish on boundaries $\partial\Omega^o$ and $\partial\Omega^i$ respectively.
Then, we multiply the first equation in (\ref{auxv}) by ${\bf v}^{o}-{\bf U}^o$ and the second equation by ${\bf v}^{i}-{\bf U}^i.$ Due to the described choice of a test-function, the integration by parts give the integral identities where there are no terms with pressure $p^o,$ $p^i:$
$$\int\limits_{\Omega^o}\nabla{\bf v}^o(\nabla{\bf v}^o-\nabla{\bf U}^o)\,dx=\int\limits_{\Omega^o}{\bf \Phi}^o({\bf v}^o-{\bf U}^o)\,dx;$$
$$m\int\limits_{\Omega^i}\nabla{\bf v}^i(\nabla{\bf v}^i-\nabla{\bf U}^i)\,dx+ms^2\int\limits_{\Omega^i}{\bf v}^i({\bf v}^i-{\bf U}^i)\,dx=\int\limits_{\Omega^i}{\bf \Phi}^i({\bf v}^i-{\bf U}^i)\,dx$$
In a standard way (see e.g. the book \cite{Lad} by  Ladyzhenskaya) we arrive at the inequalities (\ref{normsv}).
Inequalities (\ref{normspressure}) follows from Theorem 1.5 in  \cite{Solonnikov} (see also \cite{Temam}) and classical Friedrich's inequality which holds due to the assumption (\ref{pressure_mean}).
\end{proof}
Now we apply Lemma \ref{auxv} to $\Phi^o=\delta^2\triangle \varphi^o\nabla \varphi^o,$ $\Phi^i=\delta^2\triangle \varphi^i\nabla \varphi^i.$
\begin{lemma}\label{normfi}
The norms of velocity satisfy to
\begin{equation}\label{Fi}
\begin{aligned}
&\|{\bf v}^o\|^2_{L_2}+\|{\bf v}^i\|^2_{L_2}\leq  (\|(Z_{+}C^o_{+}-Z_{-}C^o_{-})^2\|_{L_{\infty}}+\|(Z_{+}C^i_{+}-Z_{-}C^i_{-}-\sigma)^2\|_{L_{\infty}})\times\\
&\left(\frac{16\left(1-1/{\gamma}\right)^2}{\delta^4}\|Z_{+}C^o_{+}-Z_{-}C^o_{-}\|^2_{L_{2}}+\frac{4}{\delta^4}\|Z_{+}C^i_{+}-Z_{-}C^i_{-}-\sigma\|^2_{L_{2}}\right)+\\
&+\|{\bf U}\|^2_{L_2}+\|{\bf U}^o\|^2_{L_2}+\|{\bf U}^i\|^2_{L_2}
\end{aligned}
\end{equation}
\end{lemma}

\begin{proof}

\begin{equation}\label{FiiFio}
\begin{aligned}
&\|\Phi^o\|^2_{L_2}+\|\Phi^i\|^2_{L_2}=\int\limits_{\Omega^o}(\delta^2\triangle \varphi^o\nabla \varphi^o)^2\,dx+\int\limits_{\Omega^i}(\delta^2\triangle \varphi^i\nabla \varphi^i)^2\,dx\leq\\
&\leq\delta^4\|(\Delta \varphi^o)^2\|_{L_{\infty}}\|\nabla \varphi^o\|^2_{L_2}+\delta^4\|(\Delta \varphi^i)^2\|_{L_{\infty}}\|\nabla \varphi^i\|^2_{L_2}\leq\\
&\leq  (\|\left(Z_{+}C^o_{+}-Z_{-}C^o_{-}\right)^2\|_{L_{\infty}}+\|Z_{+}C^i_{+}-Z_{-}C^i_{-}-\sigma^2\|_{L_{\infty}})\times\\
&\left(\frac{16\left(1-1/{\gamma}\right)^2}{\delta^4}\|Z_{+}C^o_{+}-Z_{-}C^o_{-}\|^2_{L_{2}}+\frac{4}{\delta^4}\|Z_{+}C^i_{+}-Z_{-}C^i_{-}-\sigma\|^2_{L_{2}}\right).
\end{aligned}
\end{equation}
To prove (\ref{Fi}) it remains to use Lemma \ref{L1}.
\end{proof}

\begin{lemma}\label{normpr}
The norms of pressure are bounded as follows: 
\begin{equation}\label{pr}
\begin{aligned}
&\| p^o\|^2_{L_2}\leq C_3\left(1-\frac{1}{\gamma}\right)^2
\left(\frac{16\left(1-1/{\gamma}\right)^2}{\delta^4}\|\left(Z_{+}C^o_{+}-Z_{-}C^o_{-}\right)\|^2_{L_{2}}+\frac{4}{\delta^4}\|\left(Z_{+}C^i_{+}-Z_{-}C^i_{-}-\sigma\right)\|^2_{L_{2}}\right),\\
&\| \nabla p^o\|^2_{H^{-1}}\leq C_3
\left(\frac{16\left(1-1/{\gamma}\right)^2}{\delta^4}\|\left(Z_{+}C^o_{+}-Z_{-}C^o_{-}\right)\|^2_{L_{2}}+\frac{4}{\delta^4}\|\left(Z_{+}C^i_{+}-Z_{-}C^i_{-}-\sigma\right)\|^2_{L_{2}}\right),\\
&\| p^i\|^2_{L_2}+\| \nabla p^i\|^2_{H^{-1}}\leq C_4
\left(\frac{16\left(1-1/{\gamma}\right)^2}{\delta^4}\|\left(Z_{+}C^o_{+}-Z_{-}C^o_{-}\right)\|^2_{L_{2}}+\frac{4}{\delta^4}\|\left(Z_{+}C^i_{+}-Z_{-}C^i_{-}-\sigma\right)\|^2_{L_{2}}\right).\\
\end{aligned}
\end{equation}
\end{lemma}

\begin{proof}
The result of the theorem is a direct consequence of Lemma \ref{auxv} and estimates (\ref{FiiFio}).
\end{proof}

\begin{remark}
If the norms $\|\left(Z_{+}C^o_{+}-Z_{-}C^o_{-}\right)\|_{L_{2}}$ and $\|\left(Z_{+}C^i_{+}-Z_{-}C^i_{-}-\sigma\right)\|_{L_{2}}$ have the asymptotics $O(\delta^2),$ then estimates (\ref{pr}) read as follows:

\begin{equation}\label{prdelta}
\begin{aligned}
\| p^o\|^2_{L_2}+\| \nabla p^o\|^2_{H^{-1}}\leq C_3,
\end{aligned}
\end{equation}
where the constant $C_3$ does not depend neither concentrations nor $\delta.$
\end{remark}
\section{Estimates of the permeability}

Let us estimate now the coefficient of a hydrodynamic permeability $L_{11}.$ Rewriting $L_{11}$ in the dimentionless form, one concludes that
$${L}_{11}=-\frac{{U}}{{\nabla} {p}},$$ when  $\varphi^o|_{\theta=0}=\varphi^o|_{\theta=\pi},$ $C^o|_{\theta=0}=C^o|_{\theta=\pi}$ as  $r=\frac{1}{\gamma}.$
The cell pressure gradient $\nabla p$ is defined as  $$\nabla p=-\frac{\overrightarrow{F}}{V_{cell}}.$$
Let $\Sigma$ be the stress tensor given by
$$\Sigma=p^o{\mathbf I}-\frac{1}{2}\left(\nabla {\bf v}^o+ \nabla ({\bf v}^o)^T\right),$$

then $$\overrightarrow{F}=\int\limits_{\Gamma^i}\Sigma\cdot {\overrightarrow{\bf n}}\,dS$$
Hence,

$$\|\overrightarrow{F}\|^2_{L_\infty}=\int\limits_{\Omega^o}(\nabla p^o-\Delta {\bf v}^o)^2\,dx\leq 2(\|\nabla p^o\|^2_{L_2}+\|\Delta {\bf v}^o\|^2_{L_2})$$

According to the boundary-value problem in the outer domain,

$$\|\Delta v^o\|^2_{L_2}\leq 2\left(\|\nabla p^o\|^2_{L_2}+\kappa_1\|\left(Z_{+}C^o_{+}-Z_{-}C^o_{-}\right)\|_{L_2}^2+\frac{1}{\kappa_1}\|\nabla \varphi^o\|^2_{L_2}\right),$$
where $\kappa_1>0$ is an arbitrary constant.
From the other hand, multiplying the third  equation in system (\ref{Pr_in}) by $\varphi^o,$ integrating it over $\Omega^o$ and taking into the account the assumption on zero gradient of $\varphi^o$ on the boundaries of $\Omega^o,$ we deduce that
$$\|\nabla \varphi^o\|^2_{L_2}\leq \frac{1}{\delta^2}\left(\kappa_2\|\left(Z_{+}C^o_{+}-Z_{-}C^o_{-}\right)\|_{L_2}^2+\frac{1}{\kappa_2}\|\varphi^o\|^2_{L_2}\right),$$
where $\kappa_2>0$ is an arbitrary constant.
Finally, we obtain the estimate

$$\|\overrightarrow{F}\|^2_{L_\infty}\leq 4\|\nabla p^o\|^2_{L_2}+\left(2\kappa_1+\frac{2\kappa_2}{\delta^2\kappa_1}\right)\|\left(Z_{+}C^o_{+}-Z_{-}C^o_{-}\right)\|_{L_2}^2+\frac{2}{\delta^2\kappa_1\kappa_2}\|\varphi^o\|^2_{L_2}.$$

Chose $\kappa_1=\kappa_2=\frac{1}{\delta^2}.$ Applying inequality (\ref{normsfiH1constants}) for $\varphi^o$ and estimates for pressure gradient given in Lemma 6.3,  we deduce further:
\begin{equation}\label{norma F}
\begin{aligned}
&\|\overrightarrow{F}\|^2_{L_\infty}\leq C_1\|\left(Z_{+}C^o_{+}-Z_{-}C^o_{-}\right)\|_{L_2}^2+C_2\|\left(Z_{+}C^i_{+}-Z_{-}C^i_{-}-\sigma\right)\|_{L_2}^2,\quad \\
\end{aligned}
\end{equation}
 where
 $$C_1=\max\{\frac{64}{\delta^4}\left(1-\frac{1}{\gamma}\right)^2, \left(\frac{4}{\delta^2}+\frac{2}{\delta}\right), \frac{2}{\delta^4} \},$$
$$C_2=\max\{\frac{16}{\delta^2}, \frac{2}{\delta^4} \} $$
Thus, the norm of permeability coefficient $L_{11}$ satisfies the estimates
\begin{equation}\label{norma L11}
\begin{aligned}
&\|L_{11}\|^2_{L_\infty}\leq \frac{4U}{3\gamma^3}\left(C_1\|\left(Z_{+}C^o_{+}-Z_{-}C^o_{-}\right)\|_{L_2}^2+C_2\|\left(Z_{+}C^i_{+}-Z_{-}C^i_{-}-\sigma\right)\|_{L_2}^2\right)^{-1},\quad \\
\end{aligned}
\end{equation}
Consider different limiting cases depending on $\delta.$

\begin{itemize}
\item $\|\left(Z_{+}C^o_{+}-Z_{-}C^o_{-}\right)\|_{L_2}^2=O(\delta^\alpha),\|\left(Z_{+}C^i_{+}-Z_{-}C^i_{-}-\sigma\right)\|_{L_2}^2=O(\delta^\alpha),$ \text{where}\\ ${ 0\leq \alpha\leq 4}, $ $\delta<<1:$

$$\|L_{11}\|^2_{L_\infty}\leq \frac{b^3U\delta^{4-\alpha}}{3},  \text{  if }\left(1-\frac{1}{\gamma}\right)^2<\frac{1}{32} ,$$

$$\|L_{11}\|^2_{L_\infty}\leq \frac{b^3U\delta^{4-\alpha}}{3}\left(1+32\left(1-\frac{1}{\gamma}\right)^2\right)^{-1},  \text{  if }\left(1-\frac{1}{\gamma}\right)^2>\frac{1}{32};$$

\item $\|\left(Z_{+}C^o_{+}-Z_{-}C^o_{-}\right)\|_{L_2}^2=O(\delta^\alpha),\|\left(Z_{+}C^i_{+}-Z_{-}C^i_{-}-\sigma\right)\|_{L_2}^2=O(\delta^\alpha),$ \text{where}\\ ${ 0\leq \alpha\leq 2}, $ $\delta>1:$
$$\|L_{11}\|^2_{L_\infty}\leq \frac{2b^3U\delta^{2-\alpha}}{3},  \text{  if }\left(1-\frac{1}{\gamma}\right)^2<\left(\frac{\delta^2}{16}+\frac{\delta^3}{32}\right);$$
$$\|L_{11}\|^2_{L_\infty}\leq \frac{b^3U\delta^{2-\alpha}}{12}\left(1+4\delta^{-2}\left(1-\frac{1}{\gamma}\right)^2\right)^{-1},  \text{  if }\left(1-\frac{1}{\gamma}\right)^2>\left(\frac{\delta^2}{16}+\frac{\delta^3}{32}\right);$$

\end{itemize}

\section{Asymptotics for the densities of ion's fluxes}

\begin{theorem}\label{j_estimates}
Let  $\|\left(Z_{+}C^{\kappa}_{+}-Z_{-}C^{\kappa}_{-}-\sigma^{\kappa}\right)\|_{L_2}^2=O(\delta^\alpha), \text{where}\  \alpha >0, $  $\kappa=o,i.$

Then 

$$
\left\|j^{\kappa}_{\pm}\right\|^2_{L_2}\leq \max{\left\lbrace\frac{1}{2}+\frac{1}{2\nu_{\kappa}^2Pe^2},\frac{1}{\nu_{\kappa}Pe}+\frac{1}{2\nu_{\kappa}^2Pe^2}\right\rbrace}\|C^{\kappa}_{\pm}\|^2_{H^1}+$$
$$+\frac{1}{2\nu_{\kappa}Pe}\left\|Z^{\kappa}_{\pm}\right\|^2_{L_2}+\frac{1}{\nu_{\kappa}Pe}\left\|Z^{\kappa}_{\pm}(C^{\kappa})^2_{\pm}\right\|^2_{L_2}+ $$
$$+\max{\left\lbrace\frac{1}{2}+\frac{3}{2\nu_{\kappa}Pe},C_{\kappa}\left(\frac{3}{2\nu_{\kappa}Pe}+\frac{1}{2\nu_{\kappa}^2Pe^2}\right)\right\rbrace}O(\delta^{\alpha-4})+
$$
$$+\left(\frac{1}{2}+\frac{3}{2\nu_{\kappa}Pe}\right)\left(\|U\|_{L_2}^2+\|U^o\|_{L_2}^2+\|U^i\|_{L_2}^2\right)$$

as $\delta<<1,0\leq \alpha\leq 4$ or if $\delta>1, \alpha > 4.$

Here 
$C_{\kappa}=\begin{cases}8(1-\frac{1}{\gamma})^2,  & \text{ if } \kappa=o,\\
2, & \text { if } \kappa=i.
\end{cases}$

In case $\delta>1,$  $0\leq \alpha\leq 4$  or  $\delta<<1, \alpha> 4 $ then 

$$
\left\|j^{\kappa}_{\pm}\right\|^2_{L_2}\leq \max{\left\lbrace\frac{1}{2}+\frac{1}{2\nu_{\kappa}^2Pe^2},\frac{1}{\nu_{\kappa}Pe}+\frac{1}{2\nu_{\kappa}^2Pe^2}\right\rbrace}\|C^{\kappa}_{\pm}\|^2_{H^1}+$$
$$+\frac{1}{2\nu_{\kappa}Pe}\left\|Z^{\kappa}_{\pm}\right\|^2_{L_2}+\frac{1}{\nu_{\kappa}Pe}\left\|Z^{\kappa}_{\pm}(C^{\kappa})^2_{\pm}\right\|^2_{L_2}+ $$
$$+\left(\frac{1}{2}+\frac{3}{2\nu_{\kappa}Pe}\right)\left(O(\delta^{\alpha})+\|U\|_{L_2}^2+\|U^o\|_{L_2}^2+\|U^i\|_{L_2}^2\right)$$

\end{theorem}

\begin{proof}
Integrating the squared formula for ion flux densities from (\ref{Pr_in}), one can obtain the following integral equation:

$$\int\limits_{\Omega^{\kappa}}|j^{\kappa}_{\pm}|^2\,dx=\int\limits_{\Omega^{\kappa}}|C^{\kappa}_{\pm}|^2|{\bold v}^{\kappa}|^2\,dx+\frac{1}{\nu^2_{\kappa}Pe^2}\int\limits_{\Omega^{\kappa}}|C^{\kappa}_{\pm}|^2Z^2_{\pm}|\nabla\varphi^{\kappa}|^2\,dx+\frac{1}{\nu^2_{\kappa}Pe^2}\int\limits_{\Omega^{\kappa}}|\nabla C^{\kappa}_{\pm}|^2\,dx-$$
$$-\frac{2}{\nu_{\kappa}Pe}\int\limits_{\Omega^{\kappa}}{\bold v}^{\kappa}C^{\kappa}_{\pm}\nabla C^{\kappa}_{\pm}\,dx-\frac{2}{\nu_{\kappa}Pe}\int\limits_{\Omega^{\kappa}}{\bold v}^{\kappa}(C^{\kappa}_{\pm})^2\nabla {\varphi}^{\kappa}Z_{\pm}\,dx+\frac{2}{\nu_{\kappa}Pe}\int\limits_{\Omega^{\kappa}}C^{\kappa}_{\pm}\nabla C^{\kappa}_{\pm}\nabla {\varphi}^{\kappa}Z_{\pm}\,dx.$$

By using a standard inequality $\int fg\,dx\leq \frac{1}{2}\left(\int f^2\,dx+ \int g^2\,dx\right)$ and observation that $C^{\kappa}_{\pm}\nabla C^{\kappa}_{\pm}=\frac{1}{2}|\nabla C^{\kappa}_{\pm}|^2,$ we proceed the estimation

$$\int\limits_{\Omega^{\kappa}}|j^{\kappa}_{\pm}|^2\,dx\leq\left(\frac{1}{2}+\frac{1}{2\nu^2_{\kappa}Pe^2}\right)\int\limits_{\Omega^{\kappa}}|C^{\kappa}_{\pm}|^2\,dx+\left(\frac{1}{\nu_{\kappa}Pe}+\frac{1}{2\nu^2_{\kappa}Pe^2}\right)\int\limits_{\Omega^{\kappa}}|\nabla C^{\kappa}_{\pm}|^2\,dx+$$
$$+\left(\frac{1}{2}+\frac{3}{2\nu_{\kappa}Pe}\right)\int\limits_{\Omega^{\kappa}}| {\bold v}^{\kappa}_{\pm}|^2\,dx+\left(\frac{1}{2\nu^2_{\kappa}Pe^2}+\frac{3}{2\nu_{\kappa}Pe}\right)\int\limits_{\Omega^{\kappa}}| \nabla{\varphi}^{\kappa}_{\pm}|^2\,dx+$$
$$+\frac{1}{2\nu_{\kappa}Pe}\int\limits_{\Omega^{\kappa}}Z^2_{\pm}\,dx+\frac{1}{\nu_{\kappa}Pe}\int\limits_{\Omega^{\kappa}}Z^2_{\pm}(C^{\kappa}_{\pm})^4\,dx.$$

Having in mind estimates (\ref{weakvsum}) and (\ref{normsfigradL_2}) for the velocity and electric potential, we derive

$$
\left\|j^{\kappa}_{\pm}\right\|^2_{L_2}\leq \max{\left\lbrace\frac{1}{2}+\frac{1}{2\nu_{\kappa}^2Pe^2},\frac{1}{\nu_{\kappa}Pe}+\frac{1}{2\nu_{\kappa}^2Pe^2}\right\rbrace}\|C^{\kappa}_{\pm}\|^2_{H^1}+$$
$$+\frac{1}{2\nu_{\kappa}Pe}\left\|Z^{\kappa}_{\pm}\right\|^2_{L_2}+\frac{1}{\nu_{\kappa}Pe}\left\|Z^{\kappa}_{\pm}(C^{\kappa})^2_{\pm}\right\|^2_{L_2}+ $$
$$+\left(\frac{1}{2}+\frac{3}{2\nu_{\kappa}Pe}\right)\left(\|Z_{+}C^{o}_{+}-Z_{-}C^{o}_{-}\|_{L_{\infty}}^2+\|Z_{+}C^{i}_{+}-Z_{-}C^{i}_{-}-\sigma\|_{L_{\infty}}^2\right)\times
$$
$$\times\left(\frac{16\left(1-\frac{1}{\gamma}\right)^2}{\delta^4}\|Z_{+}C^{o}_{+}-Z_{-}C^{o}_{-}\|_{L_2}^2+\frac{4}{\delta^4}\|Z_{+}C^{i}_{+}-Z_{-}C^{i}_{-}-\sigma\|_{L_2}^2+
\|U\|_{L_2}^2+\|U^o\|_{L_2}^2+\|U^i\|_{L_2}^2\right)+$$
$$+\left(\frac{1}{2\nu^2_{\kappa}Pe^2}+\frac{3}{2\nu_{\kappa}Pe}\right)\left(\frac{16\left(1-\frac{1}{\gamma}\right)^2}{\delta^4}\|Z_{+}C^{o}_{+}-Z_{-}C^{o}_{-}\|_{L_2}^2+\frac{8}{\delta^2}\|Z_{+}C^{i}_{+}-Z_{-}C^{i}_{-}-\sigma\|_{L_2}^2\right)$$

\end{proof}

\section{Concluding remarks}
The obtained estimates show the boundedness of unknown solution to the original problem and it dependence on the Debye radius, incoming velocity and geometry of the porous shell. 
In particular, we can observe that the wider the outer layer with Stokes flow, the greater value can take the permeability. Observe that the influence of the gradient concentration, the charge modules and electric potential is neglectable in the limit case when the term $\nu_{\kappa}Pe\rightarrow \infty.$  The Peclet number $Pe$ tends to the infinity as, for instance, the coefficient of diffution $D_0$ approaches to zero. 
Moreover, the different upper bound for the ion flux densities can occur depending on the ratios $\frac{\delta^{\alpha-4}}{Pe^2}$ and on $\frac{\delta^{\alpha}}{Pe}.$ If $\delta^{\alpha}<<Pe\nu_{\kappa}$ as $\delta>1,$  $0\leq \alpha\leq 4$  or  $\delta<<1, \alpha> 4, $ then the norm of ion flux densities  takes the smaller values. A similar observation holds 
for  the case $\delta<<1,0\leq \alpha\leq 4$ or if $\delta>1, \alpha > 4$ when $\delta^{\alpha-4}<<Pe^2\nu^2_{\kappa}.$

\subsection*{Acknowledgements} This work was supported by Russian Science Foundation 20-09-00670. The author thanks professor Anatoly Filippov for fruitful discussions.

{\small
{\em Tallinskaya str. 34,
Moscow, Russia}:
{\em Yulia Koroleva}, HSE (National Research University)
 e-mail: \texttt{yo.koroleva@hse.ru}.

}


{\small
\begin{thebibliography}{999}

\bibitem{Solonnikov}
V.A.Solonnikov, On boundary-value problems for linear parabolic systems of general type,{\it Trudy Mat. Inst. Steklov,} Vol.83 (1965), 1-162.
\bibitem{Temam}
R. Temam, {\it Navier-Stokes equations. Theory and Numerical Analysis,} North-Holland, Amsterdam, New-York, Oxford, 1979.
\bibitem{Lad}
Ladyzhenskaya O. A., {\it The  Mathematical Theory of Problems of Viscous Incompressible Flow,} New York : Gordon and Breach, 1969.
\bibitem{Lions}
Lions, J.L. { \it Some Methods of Solving Non-Linear Boundary Value Problems,} Paris: Dunod-Gauthier-Villars, 1969.  
\bibitem{Filippov1}
A.N. Filippov // Colloid J. 2018. V. 80. P. 716–727.
\bibitem{Filippov2}
 A.N. Filippov // Colloid J. 2018. V. 80. P. 728–738.

\bibitem{Filippov3}
A.N. Filippov // Colloid J. 2021. V. 83. (in press).

\bibitem{Filippov4}
 A.N. Filippov, S.A. Shkirskaya // Colloid J. 2019. Vol. 81. P. 597–606.

\bibitem{rel1} Dieter Bothe, Andr e Fischer, Michel Pierre, and Guillaume Rolland // Global existence for diffusion-electromigration systems in
Space dimension three and higher, {\it Nonlinear Analysis,} 2013
DOI: 10.1016/j.na.2013.12.015

\bibitem{rel2} Peter Constantin, Mihaela Ignatova,  Fizay-Noah Lee, // Nernst–Planck–Navier–Stokes Systems far
from Equilibriu, {\it Arch. Rational Mech. Anal.} 240 (2021) 1147–1168

\bibitem{rel3} Peter Constantin, Mihaela Ignatova,  Fizay-Noah Lee, // Interior Electroneutrality in
Nernst–Planck–Navier–Stokes Systems, {\it Arch. Rational Mech. Anal.} 42 (2021) 1091–1118



\end{thebibliography}}
\end{document}